\documentclass[11pt,reqno]{amsart}
\selectfont

\usepackage{enumerate}
\usepackage[page,title,titletoc,header]{appendix}
\newtheorem{theorem}{Theorem}[section]
\newtheorem{lemma}[theorem]{Lemma}
\newtheorem{corollary}[theorem]{Corollary}
\newtheorem{proposition}[theorem]{Proposition}

\newtheorem{conjecture}[theorem]{Conjecture}

\newtheorem{Concluding Remark}[theorem]{Concluding Remark}

\newcommand{\no}{\nonumber}
\numberwithin{equation}{section}

\begin{document}
\begin{center}
	{\Large\bf Asymptotic formulas for the coefficients of the truncated theta
		series}
\end{center}

\begin{center}
	
	Renrong Mao  
	
	Department of Mathematics,\\
	Soochow University, \\
	Suzhou, 215006,
	People's Republic of China\\[6pt]

	Email: rrmao@suda.edu.cn

\end{center}

\textbf{Abstract:}
Motivated by the groundbreaking work of Andrews and Merca, truncated theta series have been extensively studied over the years. 
In particular, Merca made conjectures on the non-negativity of the coefficient of $q^N$ in  
truncated series from the Jacobi triple product identity and the quintuple product identity.
In this paper, using Wright's Circle Method, we establish asymptotic formulas for the coefficients of truncated theta series and prove that Merca's conjectures are true for sufficiently large $N$.



\allowdisplaybreaks
\section{INTRODUCTION}
Euler’s pentagonal number theorem plays an important role in the theory of partitions. It gives 
\begin{align}\label{pen}
	(q;q)_\infty=\sum_{j=0}^\infty(-1)^jq^{j(3j+1)/2}(1-q^{2j+1}).
\end{align}
In equation \eqref{pen} and for the rest of this article, we use the notations
\begin{align*}
	(x_1, x_2, \ldots, x_k;q)_m &:=\prod_{n=0}^{m-1}
	(1-x_1q^n)(1-x_2q^n)\cdots(1-x_kq^n),\\
	(x_1,x_2, \ldots, x_k;q)_\infty &:=\prod_{n=0}^\infty
	(1-x_1q^n)(1-x_2q^n)\cdots(1-x_kq^n),\\
	{ L \brack K}_q&:=\left\{
	\begin{aligned}
		&0,\quad&&\textrm{if }K<0\textrm{ or } K>L,\\
		&\frac{(q;q)_L}{(q;q)_K(q;q)_{L-K}}, \quad&&\text{otherwise}\\
	\end{aligned}
	\right.
\end{align*}
and we require $|q|<1$ for absolute convergence.

Andrews and Merca \cite{tr} proved the truncated pentagonal number theorem:
\begin{align}\label{trpen1}
	\frac{1}{(q)_\infty}&\sum_{j=0}^{k-1}(-1)^{j}q^{j(3j+1)/2}(1-q^{2j+1})\\&= 1+(-1)^{k-1}\sum_{n=1}^{\infty}\frac{q^{(k+1)n+k(k-1)/2}}{(q;q)_n}\begin{bmatrix}
		&n-1\quad \\&k-1\quad
	\end{bmatrix}_q.\nonumber
\end{align}  
Their work inspired a wide study on truncated identities and their partition-theoretic interpretations. See \cite{guass,mao2,wang,xia,yao,yee,zhou}, for example.
In particular,
Andrews-Merca \cite{tr} and Guo-Zeng \cite{guass} made the following conjecture truncated Jacobi triple product
series, independently.
\begin{conjecture}\label{c1}
	For positive integers $k, R, S$ with  $k\geq1$ and $1\leq S< R/2$, let
	\begin{align*}
	\sum_{N=0}^\infty 	\mathcal{C}^\prime_{R,S,k}(N)q^N:=		(-1)^{k-1}\frac{\sum_{j=0}^{k-1}(-1)^{j}
			q^{Rj(j+1)/2-Sj}(1-q^{(2j+1)S})}
		{(q^R,q^{R-S},q^R;q^R)_\infty}.
	\end{align*}
Then $\mathcal{C}^\prime_{R,S,k}(N)\geq0$ for all $N\geq0$.
\end{conjecture}
Conjecture \ref{c1} was first proved by Yee \cite{yee} and the author \cite{mao1}, independently.
He, Ji and Zang \cite{ji} studied bilateral truncated Jacobi’s identity and generalized the result in Conjecture \ref{c1}.
Truncated Jacobi’s identities analogous to \eqref{trpen1} are obtained in \cite{wang1}.
 Merca \cite{mecra1,mecra2} proposed a stronger version of Conjecture \ref{c1}.
\begin{conjecture}\label{con1}
	For positive integers $1\leq S<R$ with $k\geq1$, let
\begin{align*}
\sum_{N=0}^\infty \mathcal{C}_{R,S,k}(N)q^N:=	&(-1)^k\frac{	\sum_{j=k}^{\infty}(-1)^{j}q^{Rj(j+1)/2}(q^{-jS}-q^{jS+S})}{(q^S,q^{R-S};q^R)_\infty}.
\end{align*}
Then $\mathcal{C}_{R,S,k}(N)\geq0$ for all $N\geq0$.
\end{conjecture}
 Ballantine and B. Feigon \cite{ba} proved Conjecture \ref{con1} with $S=1, R=3$ and $k=1,2,3.$
 Ding and Sun \cite{ding} confirmed Conjecture \ref{con1} when $R=3S.$
The first result of this paper is the following asymptotic formulas for $\mathcal{C}^\prime_{R,S,k}(N)$ and $\mathcal{C}_{R,S,k}(N)$ from which Conjectures \ref{c1} and \ref{con1} follow for sufficiently large $N$ whenever $R, S, k$ are fixed.
Without loss of generality, we assume that $(R,S)=1$ in the rest of the paper.
\begin{theorem}\label{thmain1}
For arbitrary fixed positive integers
$k, 1\leq S<R$ and $(R,S)=1$, 
we have, as $N\rightarrow\infty$, 
\begin{align}
	\mathcal{C}_{R,S,k}(N)&=\frac{\pi k S N^{-5/4}}{4(3R)^{3/4}\sin\left(\frac{S\pi }{R}\right)}e^{2\pi\sqrt{\frac{N}{3R}}}
	+O\left(N^{-3/2}e^{2\pi\sqrt{\frac{N}{3R}}}\right)\label{thmain1eq}
	\intertext{and}
		\mathcal{C}^\prime_{R,S,k}(N)&=\frac{\pi k S N^{-3/2}}{8\sqrt{2R}\sin\left(\frac{S\pi }{R}\right)}e^{2\pi\sqrt{\frac{N}{2R}}}
	+O\left(N^{-2}e^{2\pi\sqrt{\frac{N}{2R}}}\right)\label{thmain1eqp}
	\end{align}
In particular, Conjectures \ref{c1} and \ref{con1} are true for $N$ large enough. 	
	\end{theorem}
\emph{Remark}: Special cases of \eqref{thmain1eqp} with $(R,S)=(3,1), (4,1), (2,1)$ are first obtained by Chern \cite{chern0}.

 To prove Theorem \ref{thmain1}, we establish the following asymptotic formulas which have their own interests. 
\begin{theorem}\label{thmain2}
	Assume that $d$ is a non-negative integer and positive numbers $ a,c$ satisfying $aj^2+cj\in \mathbb{Z}$ for all non-negative integers $j$.
For $1\leq S<R$ with $(S,R)=1$, define \begin{align*}
		G_{a,c,d}(q):&=\sum_{j=0}^{\infty}q^{aj^2+cj+d},\\
		\sum_{N=0}^\infty \mathcal{B}_{a,c,d}(N)q^N&:=	L_{a,c,d}(q):=	\frac{		G_{a,c,d}(q)}{(q^S,q^{R-S};q^R)_\infty},\\
			\sum_{N=0}^\infty \mathcal{B}^\prime_{a,c,d}(N)q^N&:=	L_{a,c,d}^\prime(q):=	\frac{			G_{a,c,d}(q)}{(q^S,q^{R-S},q^R;q^R)_\infty}.
	\end{align*}
Then we have, as $N\rightarrow\infty$, 
\begin{align}
\mathcal{B}_{a,c,d}(N)&=\frac{\sqrt{\frac{\pi}{a}}}{4\sin\left(\frac{S\pi }{R}\right)}
	\times\left(\frac{\pi}{\sqrt{3RN}}\right)^{1/2}
	I_{-\frac{1}{2}}\left(2\pi\sqrt{\frac{N}{3R}}\right)\nonumber\\
	&\quad-\frac{B_1\left(\frac{c}{2a}\right)}{2\sin\left(\frac{S\pi }{R}\right)}\left(\frac{\pi}{\sqrt{3RN}}\right)
	I_{-1}\left(2\pi\sqrt{\frac{N}{3R}}\right)\nonumber\\
	&\quad -\frac{\sqrt{\frac{\pi}{a}}\left(d-\frac{c^2}{4a}+\frac{R}{12}-\frac{S}{2}+\frac{ S^2}{2R} \right)}{4\sin\left(\frac{S\pi }{R}\right)}
	\left(\frac{\pi}{\sqrt{3RN}}\right)^{3/2}
	I_{-\frac{3}{2}}\left(2\pi\sqrt{\frac{N}{3R}}\right)\nonumber\\
	&\quad +\frac{\left[	\left(d-\frac{c^2}{4a}+\frac{R}{12}-\frac{S}{2}+\frac{ S^2}{2R} \right)B_1\left(\frac{c}{2a}\right)+\frac{aB_3\left(\frac{c}{2a}\right)}{3}\right]}{2\sin\left(\frac{S\pi }{R}\right)}
	\left(\frac{\pi}{\sqrt{3RN}}\right)^{2}
	I_{-2}\left(2\pi\sqrt{\frac{N}{3R}}\right)\nonumber\\
	&\quad
	+O\left(N^{-\frac{3}{2}}e^{2\pi\sqrt{\frac{N}{3R}}}\right)\label{thmain2eq}
	\intertext{and}
	\mathcal{B}^\prime_{a,c,d}(N)&=\frac{\sqrt{\frac{R}{2a}}}{4\sin\left(\frac{S\pi }{R}\right)}
\left(\frac{\pi}{\sqrt{2RN}}\right)	I_{-1}\left(2\pi\sqrt{\frac{N}{2R}}\right)\nonumber\\
	&\quad-\frac{\sqrt{\frac{R}{2\pi}}B_1\left(\frac{c}{2a}\right)}{2\sin\left(\frac{S\pi }{R}\right)}\times \left(\frac{\pi}{\sqrt{2RN}}\right)^{3/2}
	I_{-3/2}\left(2\pi\sqrt{\frac{N}{2R}}\right)\nonumber\\
	&\quad-\frac{\sqrt{\frac{R}{2a}}	\left(d-\frac{c^2}{4a}+\frac{R}{8}-\frac{S}{2}+\frac{ S^2}{2R} \right)}{4\sin\left(\frac{S\pi }{R}\right)}\times \left(\frac{\pi}{\sqrt{2RN}}\right)^{2}	I_{-2}\left(2\pi\sqrt{\frac{N}{2R}}\right)
	\nonumber\\ &\quad+\frac{\left[	\left(d-\frac{c^2}{4a}+\frac{R}{8}-\frac{S}{2}+\frac{ S^2}{2R} \right)B_1\left(\frac{c}{2a}\right)+\frac{aB_3\left(\frac{c}{2a}\right)}{3}\right]\sqrt{\frac{R}{2\pi}}}{2\sin\left(\frac{S\pi }{R}\right)}
\no	\\&\quad\times \left(\frac{\pi}{\sqrt{2RN}}\right)^{5/2}	I_{-5/2}\left(2\pi\sqrt{\frac{N}{2R}}\right)
	+O\left(N^{-\frac{7}{4}}e^{2\pi\sqrt{\frac{N}{2R}}}\right),\label{thmain2eqp}
\end{align}
where $I_{\nu}(x)$ is the modified Bessel function and $B_n(x)$ is the $n$-th Bernoulli polynomial, defined by
$$\frac{te^{xt}}{e^t-1}=\sum_{n\geq0}B_n(x)\frac{t^n}{n!}.$$
\end{theorem}
To prove Theorem \ref{thmain2}, we apply a
variant of the Hardy–Ramanujan Circle Method
due to E. Wright \cite{w1,w2}. It is recently widely used in combinatorial analysis. See \cite{kk,kkin,bmr1,dou,maospt}, for example.
In particular, the author \cite{maok} applied Wright's Circle Method to establish asymptotic formulas for $k$-ranks of partitions whose generating function is given by 
\begin{align*}
	\frac{1}
	{(q;q)_\infty}\sum_{n=1}^{\infty}(-1)^{n-1}q^{n((2k-1)n-1)/2+ln}.
\end{align*}
One can use \eqref{thmain2eqp} with $(R,S)=(3,1)$ to recover the results in \cite{maok}. 
We will follow the presentation of \cite{maok} closely.
More recently,  Chern and Xia \cite{chern} used Wright's Circle Method to establish the following asymptotic formula, for $k\geq2$ and $k=o(n^{\frac{1}{12}})$, we have, as $n\rightarrow\infty$, 
 \begin{align}
 	a_k(n)&=\frac{\pi^2}{48\sqrt{3}}(k^3-k)n^{-2}e^{\frac{2\pi \sqrt{n}}{\sqrt{6}}}+O(k^6n^{-\frac{9}{4}}e^{\frac{2\pi \sqrt{n}}{\sqrt{6}}}),\label{akn}
 	\intertext{where $a_k(n)$ is defined by}
 	\sum_{n=0}^\infty a_k(n)q^n:&=\frac{(-1)^{k-1}}{(q;q)_\infty}\sum_{j=0}^{k-1}(-1)^j
 	q^{\frac{j(3j+1)}{2}}(1-q^{\frac{j(j+1)}{2}})(1-q^{2j+1}).\label{cherxia}
 	\end{align}
 As applications, they used \eqref{akn} to confirm conjectures of Andrews-Merca \cite{tram} and Merca-Yee \cite{merca3} for sufficiently large $n$. 
 The generating functions studied in this paper are quite different from that in \eqref{cherxia}. The products are more general and we consider infinite sums in Theorem \ref{thmain2}.
To apply Wright's Circle Method, the key is to study the asymptotic behavior of $L_{a,c,d}(q)$.
 The main difficulties come from the estimates of $G_{a,c,d}(q)$ (near $q=1$) and $\frac{1}{(q^A;q^B)_\infty}$ (when $q$ away from $1$) which are given in Corollary \ref{led1}, Lemma \ref{caw}, respectively.

As another application of Theorem \ref{thmain2}, we consider truncated series from the quintuple product identity.
Chan, Ho and the author \cite{chan} proved the non-negativity of the coefficients of two truncated series from quintuple product identity.
Merca \cite{mecra2} conjectured a stronger result:
%
\begin{conjecture} \label{con3}
	For $1 \leq S<R / 2$ and $k \geq 0$, define
		\begin{align*}
	\sum_{N=0}^\infty \mathcal{D}_{R,S,k}(N)q^N:=	& \left(q^R ; q^R\right)_{\infty}\left(q^{R-2 S}, q^{R+2 S} ; q^{2 R}\right)_{\infty} \\
		&\times \left(\frac{\sum_{n=-k}^k q^{n(3 n+1) R / 2}\left(q^{-3 n S}-q^{(3 n+1) S}\right)}{\left(q^S, q^{R-S}, q^R ; q^R\right)_{\infty}\left(q^{R-2 S}, q^{R+2 S} ; q^{2 R}\right)_{\infty}}-1\right),\\
			\sum_{N=0}^\infty \mathcal{D}^\prime_{R,S,k}(N)q^N:=	& \left(q^R ; q^R\right)_{\infty}\left(q^{R-2 S}, q^{R+2 S} ; q^{2 R}\right)_{\infty} \\
		& \left(\frac{\sum_{n=-k}^{k-1} q^{n(3 n+1) R / 2}\left(q^{-3 n S}-q^{(3 n+1) S}\right)}{\left(q^S, q^{R-S}, q^R ; q^R\right)_{\infty}\left(q^{R-2 S}, q^{R+2 S} ; q^{2 R}\right)_{\infty}}-1\right).
	\end{align*}
Then, for all $N\geq0$, $\mathcal{D}_{R,S,k}(N)\geq 0$ with $k\geq0$ and $\mathcal{D}^\prime_{R,S,k}(N)\leq0$ with $k\geq1.$ 
\end{conjecture}
Applying \eqref{thmain2eq}, we find the following asymptotic formulas for $\mathcal{D}_{R,S,k}(N)$ and $\mathcal{D}^\prime_{R,S,k}(N)$. 
\begin{theorem}\label{thmain3}
	For arbitrary fixed positive integers
	$1 \leq S<R / 2$ and $(R,S)=1$, 
	we have, as $N\rightarrow\infty$, 
	\begin{align}
		\mathcal{D}_{R,S,k}(N)&=\frac{\pi (2k+1) S N^{-5/4}}{4(3R)^{3/4}\sin\left(\frac{S\pi }{R}\right)}e^{2\pi\sqrt{\frac{N}{3R}}}
		+O\left(N^{-3/2}e^{2\pi\sqrt{\frac{N}{3R}}}\right)\label{thmain3eq}
		\intertext{with $k\geq0$ and}
			\mathcal{D}^\prime_{R,S,k}(N)&=-\frac{\pi k S N^{-5/4}}{(3R)^{3/4}\sin\left(\frac{S\pi }{R}\right)}e^{2\pi\sqrt{\frac{N}{3R}}}
			+O\left(N^{-3/2}e^{2\pi\sqrt{\frac{N}{3R}}}\right)\label{thmain3eqp}
	\end{align}
with $k\geq1.$
	In particular, Conjectures \ref{con3} is true for $N$ large enough. 	
\end{theorem}

This article is organized as follows. We provide proofs of Theorems \ref{thmain1} and \ref{thmain3} in Section \ref{sth1}.
Since the proofs of \eqref{thmain2eq} and \eqref{thmain2eqp} are completely similar, we only give detailed proof of \eqref{thmain2eq}. 
In Section \ref{zh1}, we study the asymptotic behavior of $L_{a,c,d}(q)$.
We use Wright's Circle Method to complete the proof of Theorem \ref{thmain2} in Section \ref{zh2} .

\section{Proofs of Theorems \ref{thmain1} and \ref{thmain3}}\label{sth1}
In this section, we use Theorem \ref{thmain2} to prove Theorems \ref{thmain1} and \ref{thmain3}.
 
 \subsection{Proof of \eqref{thmain1eq}.} Note that
 \begin{align*}
 	&	\sum_{j=k}^{\infty}(-1)^{j}q^{Rj(j+1)/2}(q^{-jS}-q^{jS+S})\\&=(-1)^kq^{\frac{Rk(k+1)}{2}-Sk}	\sum_{j=0}^{\infty}(-1)^jq^{Rj(j+1)/2+(Rk-S)j}(1-q^{(2k+1)S+2Sj})\\
 	&=(-1)^kq^{\frac{Rk(k+1)}{2}-Sk}	\sum_{j=0}^{\infty}q^{Rj(2j+1)+2(Rk-S)j}(1-q^{(2k+1)S+4Sj})\\&\quad-(-1)^kq^{\frac{Rk(k+1)}{2}-Sk}	\sum_{j=0}^{\infty}q^{R(2j+1)(j+1)+(Rk-S)(2j+1)}(1-q^{(2k+1)S+2S(2j+1)})\\
 	&=(-1)^k\Bigg[G_{2R,(2k+1)R-2S,T_1}(q)-G_{2R,(2k+1)R+2S,T_2}(q)\\&\qquad\qquad-G_{2R,(2k+3)R-2S,T_3}(q)+G_{2R,(2k+3)R+2S,T_4}(q)\Bigg]
 \end{align*}
with 
\begin{align}
	T_1:&=\frac{Rk(k+1)}{2}-Sk, \label{t1}
\\
	T_2:&=\frac{Rk(k+1)}{2}+(k+1)S,\label{t2}\\
		T_3:&=\frac{R(k+2)(k+1)}{2}-S(k+1) \label{t3}
	\intertext{and}
	T_4:&=\frac{R(k+2)(k+1)}{2}+(k+2)S.\label{t4}
	\end{align}
 It follows
 \begin{align}
 	\sum_{n=0}^\infty \mathcal{C}_{R,S,k}(N)q^N&=	(-1)^k\frac{	\sum_{j=k}^{\infty}(-1)^{j}q^{Rj(j+1)/2}(q^{-jS}-q^{jS+S})}{(q^S,q^{R-S};q^R)_\infty}\no
 	\\&=	L_{2R,(2k+1)R-2S,T_1}(q)-L_{2R,(2k+1)R+2S,T_2}(q)\no\\&\qquad\qquad-L_{2R,(2k+3)R-2S,T_3}(q)+L_{2R,(2k+3)R+2S,T_4}(q)\no.
 \end{align}
Then
\begin{align}
	\mathcal{C}_{R,S,k}(N)&=		\mathcal{B}_{2R,(2k+1)R-2S,T_1}(N)-	\mathcal{B}_{2R,(2k+1)R+2S,T_2}(N)\no\\&\quad-	\mathcal{B}_{2R,(2k+3)R-2S,T_3}(N)+	\mathcal{B}_{2R,(2k+3)R+2S,T_4}(N)\no.
\end{align}
Invoking \eqref{thmain2eq} and noting that
\begin{align*}
	B_1(x)=x-\frac{1}{2},\qquad B_3(x)=x^3-\frac{3}{2}x^2+\frac{x}{2},
	\end{align*} after simplifying, we find that
\begin{align}
	\mathcal{C}_{R,S,k}(N)
	 &=\frac{kS}{2\sin\left(\frac{S\pi }{R}\right)}\times \left(\frac{\pi}{\sqrt{3RN}}\right)^{2}
	 I_{-2}\left(2\pi\sqrt{\frac{N}{3R}}\right)\nonumber\\
	 &\quad
	 +O\left(N^{-3/2}e^{2\pi\sqrt{\frac{N}{3R}}}\right),\label{pthmain2eq}
	\end{align}
as $N\rightarrow\infty$.

By \cite[Eq (4.12.7)]{ar}, we know that, as $x\rightarrow\infty$ (which hold for any index $\nu$),
\begin{align}\label{besl}
	I_{\nu}(x)=\frac{e^{x}}{\sqrt{2\pi x}}+O\left(\frac{e^{x}}{x^{\frac{3}{2}}}\right).
\end{align}
Replacing $x$ by $2\pi\sqrt{\frac{N}{3R}}$ gives
\begin{align}
	I_{\nu}\left(2\pi\sqrt{\frac{N}{3R}}
	\right)=\frac{(3R)^{1/4}N^{-1/4}e^{2\pi\sqrt{\frac{N}{3R}}}}{2\pi}
	+O\left(N^{-3/4}e^{2\pi\sqrt{\frac{N}{3R}}}\right)\quad(\textrm{as } N\rightarrow\infty)\label{eqi}.
\end{align}
This together with \eqref{pthmain2eq} implies \eqref{thmain1eq}
.

\subsection{Proof of \eqref{thmain1eqp}.}
From the proof of \cite[Theorem 1.2]{mao1},
 we find that
	\begin{align*}
	 	\sum_{n=0}^\infty \mathcal{C}^\prime_{R,S,k}(N)(N)q^N	&=(-1)^{k-1}\frac{\sum_{j=0}^{k-1}(-1)^{j}
			q^{Rj(j+1)/2-Sj}\left(1-q^{(2j+1)S}\right)}
		{(q^R,q^{R-S},q^R;q^R)_\infty}\\&=(-1)^{k-1}+(-1)^{k}
		\frac{\sum_{j=k}^{\infty}(-1)^{j}
			q^{Rj(j+1)/2-Sj}\left(1-q^{(2j+1)S}\right)}
		{(q^R,q^{R-S},q^R;q^R)_\infty}
	.
	\end{align*}
Proceed as in the proof of \eqref{thmain1eq} to obtain 
	\begin{align*}
	\sum_{n=0}^\infty \mathcal{C}^\prime_{R,S,k}(N)(N)q^N		&=(-1)^{k-1}+	L^\prime_{2R,(2k+1)R-2S,T_1}(q)-L^\prime_{2R,(2k+1)R+2S,T_2}(q)\no\\&\quad-L^\prime_{2R,(2k+3)R-2S,T_3}(q)+L^\prime_{2R,(2k+3)R+2S,T_4}(q)\no
	,
\end{align*}
where $T_1, T_2, T_3, T_4$ are given in \eqref{t1}-\eqref{t4}, respectively.
Then, for $N\geq1$, we have
\begin{align}
	\mathcal{C}^\prime_{R,S,k}(N)&=		\mathcal{B}^\prime_{2R,(2k+1)R-2S,T_1}(N)-	\mathcal{B}^\prime_{2R,(2k+1)R+2S,T_2}(N)\no\\&\quad-	\mathcal{B}^\prime_{2R,(2k+3)R-2S,T_3}(N)+	\mathcal{B}^\prime_{2R,(2k+3)R+2S,T_4}(N)\no.
\end{align}
Invoke \eqref{thmain2eqp} and simplify to obtain
\begin{align}
	\mathcal{C}^\prime_{R,S,k}(N)
&=	\frac{kS\sqrt{\frac{R}{2\pi}}}{2\sin\left(\frac{S\pi }{R}\right)}\times \left(\frac{\pi}{\sqrt{2RN}}\right)^{5/2}
	I_{-5/2}\left(2\pi\sqrt{\frac{N}{2R}}\right)\nonumber\\
	&\quad
	+O\left(N^{-7/4}e^{2\pi\sqrt{\frac{N}{2R}}}\right),\no
\end{align}
which together with \eqref{besl} gives \eqref{thmain1eqp}.

\subsection{Proof of \eqref{thmain3eq}.}
The quintuple product identity gives 
\begin{align}
&	\no\sum_{n=-\infty}^\infty q^{n(3 n+1) R / 2}\left(q^{-3 n S}-q^{(3 n+1) S}\right)\\&=\left(q^S, q^{R-S}, q^R ; q^R\right)_{\infty}\left(q^{R-2 S}, q^{R+2 S} ; q^{2 R}\right)_{\infty} \no\qquad\qquad\qquad\textrm{(see \cite[Eq. (1.5)]{chan}).}
	\end{align}
It follows 
	\begin{align*}
	\sum_{n=0}^\infty \mathcal{D}_{R,S,k}(N)q^N=	& \left(q^R ; q^R\right)_{\infty}\left(q^{R-2 S}, q^{R+2 S} ; q^{2 R}\right)_{\infty} \\
	&\times \left(\frac{\sum_{n=-k}^k q^{n(3 n+1) R / 2}\left(q^{-3 n S}-q^{(3 n+1) S}\right)}{\left(q^S, q^{R-S}, q^R ; q^R\right)_{\infty}\left(q^{R-2 S}, q^{R+2 S} ; q^{2 R}\right)_{\infty}}-1\right)\\=	& 
-\frac{\sum_{n=-\infty}^{-k-1} q^{n(3 n+1) R / 2}\left(q^{-3 n S}-q^{(3 n+1) S}\right)}{\left(q^S, q^{R-S}; q^R\right)_{\infty}}
\\&\quad-\frac{\sum_{n=k+1}^\infty q^{n(3 n+1) R / 2}\left(q^{-3 n S}-q^{(3 n+1) S}\right)}{\left(q^S, q^{R-S}; q^R\right)_{\infty}}.
\end{align*}
Note that
	\begin{align*}
	&	\sum_{n=-\infty}^{-k-1} q^{n(3 n+1) R / 2}\left(q^{-3 n S}-q^{(3 n+1) S}\right)+\sum_{n=k+1}^\infty q^{n(3 n+1) R / 2}\left(q^{-3 n S}-q^{(3 n+1) S}\right)\\
		&=G_{\frac{3R}{2},\frac{(6k+5)R}{2}+3S,H_1}(q)-G_{\frac{3R}{2},\frac{(6k+5)R}{2}-3S,H_2}(q)+G_{\frac{3R}{2},\frac{(6k+7)R}{2}-3S,H_3}(q)\\&\quad-G_{\frac{3R}{2},\frac{(6k+7)R}{2}+3S,H_4}(q)
\end{align*}
with 
\begin{align*}
	H_1:&=\frac{R(3k+2)(k+1)}{2}+S(3k+3), 
	\\
	H_2:&=\frac{R(3k+2)(k+1)}{2}-S(3k+2),
	\\
	H_3:&=\frac{R(3k+4)(k+1)}{2}-S(3k+3)
	\intertext{and}
	H_4:&=\frac{R(3k+4)(k+1)}{2}+S(3k+4).
\end{align*}
Then
\begin{align*}
	\sum_{n=0}^\infty \mathcal{D}_{R,S,k}(N)q^N&=	-L_{\frac{3R}{2},\frac{(6k+5)R}{2}+3S,H_1}(q)+L_{\frac{3R}{2},\frac{(6k+5)R}{2}-3S,H_2}(q)\\&\quad-L_{\frac{3R}{2},\frac{(6k+7)R}{2}-3S,H_3}(q)+L_{\frac{3R}{2},\frac{(6k+7)R}{2}+3S,H_4}(q),
\end{align*}
which gives
\begin{align*}
 \mathcal{D}_{R,S,k}(N)&=	-\mathcal{B}_{\frac{3R}{2},\frac{(6k+5)R}{2}+3S,H_1}(N)+\mathcal{B}_{\frac{3R}{2},\frac{(6k+5)R}{2}-3S,H_2}(N)\\&\quad-\mathcal{B}_{\frac{3R}{2},\frac{(6k+7)R}{2}-3S,H_3}(N)+\mathcal{B}_{\frac{3R}{2},\frac{(6k+7)R}{2}+3S,H_4}(N).
\end{align*}
Invoking \eqref{thmain2eq} and simplifying yields
\begin{align*}
	\mathcal{D}_{R,S,k}(N)&=\frac{(2k+1)S}{2\sin\left(\frac{S\pi }{R}\right)}\times \left(\frac{\pi}{\sqrt{3RN}}\right)^{2}
	I_{-2}\left(2\pi\sqrt{\frac{N}{3R}}\right)\nonumber\\
	&\quad
	+O\left(N^{-3/2}e^{2\pi\sqrt{\frac{N}{3R}}}\right),
\end{align*}
which together with \eqref{eqi} gives \eqref{thmain3eq}.
\subsection{Proof of \eqref{thmain3eqp}}
Proceeding as in the proof of \eqref{thmain3eq}, we find that
\begin{align*}
	\mathcal{D}^\prime_{R,S,k}(N)&=	-\mathcal{B}_{\frac{3R}{2},\frac{(6k+5)R}{2}+3S,H_1}(N)+\mathcal{B}_{\frac{3R}{2},\frac{(6k+5)R}{2}-3S,H_2}(N)\\&\quad-\mathcal{B}_{\frac{3R}{2},\frac{(6k+1)R}{2}-3S,H^\prime_3}(N)+\mathcal{B}_{\frac{3R}{2},\frac{(6k+1)R}{2}+3S,H^\prime_4}(N)
\end{align*}
with
\begin{align}
	H^\prime_3:&=\frac{Rk(3k+1)}{2}-3kS \no
	\intertext{and}
	H^\prime_4:&=\frac{Rk(3k+1)}{2}+S(3k+1).\no
\end{align}
Apply \eqref{thmain2eq} to obtain
\begin{align*}
	\mathcal{D}^\prime_{R,S,k}(N)&=-\frac{2kS}{\sin\left(\frac{S\pi }{R}\right)}\times \left(\frac{\pi}{\sqrt{3RN}}\right)^{2}
	I_{-2}\left(2\pi\sqrt{\frac{N}{3R}}\right)\nonumber\\
	&\quad
	+O\left(N^{-3/2}e^{2\pi\sqrt{\frac{N}{3R}}}\right),
\end{align*}
which together with \eqref{eqi} gives \eqref{thmain3eqp}.
\section{ASYMPTOTIC BEHAVIOR OF $L_{a,c,d}(q)$}\label{zh1}
In this section, we study the asymptotic behavior of $L_{a,c,d}(q)$ when $q$ approaches roots of unity. Set $q=e^{2\pi i\tau}$ with $\tau=x+iy$ and $y>0$.
Since the asymptotic behavior is largely controlled by the exponential singularities of $(q^S,q^R;q^R)_\infty^{-1}$,
the dominant pole is at $q=1$.

\subsection{Asymptotic behavior of $L_{a,c,d}(q)$ near the dominant pole.}\label{ab}
First, we consider the asymptotic behavior of $G_{a,c,d}(q)$ near $q=1$. 
Recall \cite[Theorem 1.2]{za}:
\begin{theorem}\label{thbr}
Suppose that $0 \leq \mu<\frac{\pi}{2}$ and let $D_\mu:=\left\{r \mathrm{e}^{i \alpha}: r \geq 0\right.$ and $\left.|\alpha| \leq \mu\right\}$. Let $f: \mathbb{C} \rightarrow \mathbb{C}$ be holomorphic in a domain containing $D_\theta$, so that in particular $f$ is holomorphic at the origin, and assume that $f$ and all of its derivatives are of sufficient
decay. Then for $a \in \mathbb{R}$ and $N \in \mathbb{N}_0$,
\begin{align}
\sum_{m \geq 0} f(w(m+a))=\frac{1}{w} \int_0^{\infty} f(x) \mathrm{d} x-\sum_{n=0}^{N-1} \frac{B_{n+1}(a) f^{(n)}(0)}{(n+1)!} w^n+O_N\left(w^N\right),\label{asbr}
\end{align}
uniformly, as $w \rightarrow 0$ in $D_\mu$. 
\end{theorem} 
For real numbers $b,\beta$ and $\gamma>0$, we define
\begin{align}\label{af1}
	F_b(\theta):=\sum_{n=0}^{\infty}e^{-(n^2+bn)\theta},
\end{align}
where $\theta=\gamma+\beta i$.
Proceeding as in the proof \cite[Eq. (1.5)]{maora}, 
 we can apply Theorem \ref{thbr} with $\mu=\frac{\pi}{4}$ to obtain the following asymptotic expansion of $F_b(\theta)$.
\begin{proposition}\label{tha}
	For $b>0$ and $\theta=\gamma+\beta i$ satisfying $|\beta|\leq\gamma$, and as $\gamma\rightarrow0^+$, we have
 \begin{align}
	F_b(\theta)=e^{\frac{b^2\theta}{4}}\times\Bigg\{
\frac{	\sqrt{\pi/\theta}}{2}-\sum_{n=0}^{N-1}\frac{(-1)^nB_{2n+1}(b/2)}{(2n+1)n!}\theta^{n}
	+O(\theta^N)\Bigg\}.\label{thg2eq2}
\end{align}
\end{proposition}
As application, we obtain the asymptotic behavior of $G_{a,c,d}(q)$ near $q=1$. 
\begin{corollary}\label{led1}
	For $y=\frac{1}{2\sqrt{3RN}}$ and $|x|\leq y$, as $N\rightarrow\infty$, we have
\begin{align}
	G_{a,c,d}(q)&=q^{d-\frac{c^2}{4a}}\times\Bigg\{\frac{1}{2}\sqrt{\frac{\pi}{a}}	(-2\pi i\tau)^{-\frac{1}{2}}-\sum_{n=0}^{3}\frac{(-1)^nB_{2n+1}\left(\frac{c}{2a}\right)a^n}{(2n+1)n!}(-2\pi i\tau)^{n}
	+O(N^{-2})\Bigg\}.\label{bskl}
\end{align}
\end{corollary}
\begin{proof}
	Let $\theta=-2a\pi i\tau=2a\pi y-2a\pi xi$.
	By assumption, we have $y>0$ and $|x|\leq y$ which imply $2a\pi y>0$ and $|2a\pi x|\leq2a\pi y$.
	Hence, applying \eqref{thg2eq2} with $\theta$ and $b$ replaced by $-2a\pi i\tau$ and $\frac{c}{a}$, respectively, we find that,
	as $|\tau|\rightarrow0^+$,
	\begin{align}\label{app1}   
		G_{a,c,d}(q)&=\sum_{j=0}^{\infty}q^{aj^2+cj+d}=q^d F_{\frac{c}{a}}(-2a\pi i\tau)
	\end{align}
	Since $\tau=x+yi$ with $|x|\leq y$ and $y=\frac{1}{2\sqrt{3RN}}$, we have $|\tau|^2\leq\frac{1}{6RN}$. This together with \eqref{app1} implies
	\eqref{bskl}.
\end{proof}
Next, we consider the asymptotic behavior of the infinite product in $L_{a,c,d}(q)$ near $q=1$.
\begin{lemma}\label{led2}
	For $y=\frac{1}{2\sqrt{3RN}}$ and $|x|\leq y$, as $N\rightarrow\infty$, we have
	\begin{align}
\frac{1}{(q^S,q^{R-S};q^R)_\infty}
	&=	\frac{e^{\pi i\tau\left(\frac{R}{6}-S+\frac{ S^2}{R} \right)}e^{\frac{\pi i}{6R\tau}}}{2\sin\left(\frac{S\pi }{R}\right)}
	\times \left\{1+O\left(e^{\frac{-2\pi  i}{R\tau}}\right)\right\}.\label{812914}
\end{align}
\end{lemma}
\begin{proof}
	Rewrite the infinite product on the left side of \eqref{812914} as follows:
	\begin{align}
		\frac{1}{(q^S,q^{R-S};q^R)_\infty}=-iq^{\frac{R}{12}-\frac{S}{2}}\frac{\eta\left(R\tau\right)	}{\theta(S\tau;R\tau)},\label{812fra}
	\end{align}
	where the Jacobi $\theta$-function and Dedekind-$\eta$ function are given by
	\begin{align}
		\eta(\tau):&=q^{\frac{1}{24}}\prod_{k=1}^\infty(1-q^k)\nonumber\\
		\theta(z;\tau):&=-iq^{\frac{1}{8}}\zeta^{-\frac{1}{2}}\prod_{n=1}^\infty(1-q^n)(1-\zeta q^{n-1})(1-\zeta^{-1}q^{n})\nonumber
	\end{align}
	with $\zeta:=e^{2\pi iz}$.
	We also have the following transformation laws (see e.g. \cite{rad}):
	\begin{align}
		\eta\left(-\frac{1}{\tau}\right)&=\sqrt{-i\tau}\eta(\tau),\no\\
		\theta\left(\frac{z}{\tau};-\frac{1}{\tau}\right)&=-i\sqrt{-i\tau}e^{\frac{\pi i z^2}{\tau}}\theta(z;\tau)\no
		.
	\end{align}
	It follows 
	\begin{align}
		\eta(R\tau)&=	\frac{	\eta\left(-\frac{1}{R\tau}\right)}{\sqrt{-iR\tau}},\label{812treta}\\
		\theta(S\tau;R\tau)&=\frac{	\theta\left(\frac{S}{R};-\frac{1}{R\tau}\right)}{-i\sqrt{-iR\tau}e^{\frac{\pi i S^2 \tau}{R}}}\label{812trtha}
		.
	\end{align}
	Substituting \eqref{812treta} and \eqref{812trtha} into \eqref{812fra} yields
	\begin{align}
		&\no	\frac{1}{(q^S,q^{R-S};q^R)_\infty}\\&=-q^{\frac{R}{12}-\frac{S}{2}}e^{\frac{\pi iS^2\tau}{R}}	\frac{	\eta\left(-\frac{1}{R\tau}\right)}{\theta\left(\frac{S}{R};-\frac{1}{R\tau}\right)}\no\\
		&=	\frac{-q^{\frac{R}{12}-\frac{S}{2}}e^{\frac{\pi iS^2\tau}{R} e^{\frac{-\pi i}{12R\tau}}}}{-i e^{\frac{-\pi i}{4R\tau}}e^{\frac{-S\pi i}{R}}}\no
		\\&\quad	\times \prod_{j=0}^\infty \frac{1}{\left(1-e^{\frac{2S\pi i}{R}}e^{\frac{-2\pi j i}{R\tau}}\right)\left(1-e^{\frac{-2S\pi i}{R}}e^{\frac{-2\pi (j+1) i}{R\tau}}\right)}\no
		\\
		&=	\frac{e^{\pi i\tau\left(\frac{R}{6}-S+\frac{ S^2}{R} \right)}e^{\frac{\pi i}{6R\tau}}}{2\sin\left(\frac{S\pi }{R}\right)}\no
		\\&\quad	\times \prod_{j=0}^\infty \frac{1}{\left(1-e^{\frac{2S\pi i}{R}}e^{\frac{-2\pi (j+1) i}{R\tau}}\right)\left(1-e^{\frac{-2S\pi i}{R}}e^{\frac{-2\pi (j+1) i}{R\tau}}\right)}\no
		\\
		&=	\frac{e^{\pi i\tau\left(\frac{R}{6}-S+\frac{ S^2}{R} \right)}e^{\frac{\pi i}{6R\tau}}}{2\sin\left(\frac{S\pi }{R}\right)}.\no
		\times \left\{1+O\left(e^{\frac{-2\pi  i}{R\tau}}\right)\right\}
	\end{align}
	This proves \eqref{812914}.
\end{proof}
Note that, for $y=\frac{1}{2\sqrt{3RN}}$ and $|x|\leq y$, as $N\rightarrow\infty$, we have
\begin{align*}
	e^{\frac{\pi i}{6R\tau}}&=O\left(e^{\pi\sqrt{\frac{N}{3R}}}\right).
\end{align*}
Then the following corollary follows.
\begin{corollary}
	For $y=\frac{1}{2\sqrt{3RN}}$ and $|x|\leq y$, as $N\rightarrow\infty$, we have
	\begin{align}
		L_{a,c,d}(q)&=	\frac{e^{\frac{\pi i}{6R\tau}}}{2\sin\left(\frac{S\pi }{R}\right)}
		\times \Bigg\{\frac{1}{2}\sqrt{\frac{\pi}{a}}	(-2\pi i\tau)^{-\frac{1}{2}}-B_1\left(\frac{c}{2a}\right)\no\\&\qquad\qquad-\frac{	\left(d-\frac{c^2}{4a}+\frac{R}{12}-\frac{S}{2}+\frac{ S^2}{2R} \right)}{2}\sqrt{\frac{\pi}{a}}	(-2\pi i\tau)^{\frac{1}{2}}\no\\&\qquad\qquad -\left[	\left(d-\frac{c^2}{4a}+\frac{R}{12}-\frac{S}{2}+\frac{ S^2}{2R} \right)B_1\left(\frac{c}{2a}\right)+\frac{aB_3\left(\frac{c}{2a}\right)}{3}\right]2\pi i\tau\no\\&\qquad\qquad+\zeta_1^\ast\tau^{3/2}+\zeta_2^\ast\tau^{2}+\zeta_3^\ast\tau^{5/2}\Bigg\}+O\left(e^{\pi\sqrt{\frac{N}{3R}}}N^{-\frac{3}{2}}\right), \label{8131410}
	\end{align}
	where $\zeta_1^\ast,\zeta_2^\ast,\zeta_3^\ast$ are constants depending on $R, S, a, c$ and $d$.
\end{corollary}
\begin{proof}
Corollary \ref{led1} and Lemma \ref{led2} imply that, for $y=\frac{1}{2\sqrt{3RN}}$ and $|x|\leq y$, as $N\rightarrow\infty$, we have
\begin{align*}
	L_{a,c,d}(q)&=\frac{		G_{a,c,d}(q)}{(q^S,q^{R-S};q^R)_\infty}
	\\&=\frac{e^{2\pi i\tau\left(d-\frac{c^2}{4a}+\frac{R}{12}-\frac{S}{2}+\frac{ S^2}{2R} \right)}e^{\frac{\pi i}{6R\tau}}}{2\sin\left(\frac{S\pi }{R}\right)}
\\&\quad	\times \Bigg\{\frac{1}{2}\sqrt{\frac{\pi}{a}}	(-2\pi i\tau)^{-\frac{1}{2}}-\sum_{n=0}^{3}\frac{(-1)^nB_{2n+1}\left(\frac{c}{2a}\right)a^n}{(2n+1)n!}(-2\pi i\tau)^{n}
	+O(N^{-2})\Bigg\},
\end{align*}	
which together with 
\begin{align}
e^{2\pi i\tau\left(d-\frac{c^2}{4a}+\frac{R}{12}-\frac{S}{2}+\frac{ S^2}{2R} \right)}=\sum_{n=0}^{3}\frac{\left[2\pi i\tau\left(d-\frac{c^2}{4a}+\frac{R}{12}-\frac{S}{2}+\frac{ S^2}{2R} \right)\right]^n}{n!}	+O(N^{-2})\no
\end{align}	
gives \eqref{8131410}.
	\end{proof}
\emph{Remark}: The constants $\zeta_1^\ast,\zeta_2^\ast,\zeta_3^\ast$ can be explicitly expressed by $R, S, a, c, d$ and those terms containing them will be absorbed into the error term when we use 
Wright’s Circle Method. 

\subsection{Bounds away from the dominant pole.} We first obtain the following bound of $G_{a,c,d}(q)$.
\begin{lemma}\label{paw}
  If $y=\frac{1}{2\sqrt{3RN}}$, then, as $N\rightarrow\infty$, we have $|G_{a,c,d}(q)|=O\left(\sqrt{N}\right)$.
\end{lemma}
\begin{proof}
For $q=e^{2\pi i\tau}$, where $\tau=x+\frac{1}{2\sqrt{3RN}}i$, as $N\rightarrow\infty$, we have
\begin{align*}
  |G_{a,c,d}(q)|&=\left|\sum_{j=0}^{\infty}q^{aj^2+cj+d}\right|\\&
 \leq\sum_{j=0}^{\infty}\left|q^{aj^2+cj+d}\right|\\&
 \leq\sum_{n=0}^{\infty}|q^{n}|
 \leq\frac{1}{1-|q|}
 =\frac{1}{1-e^{-\frac{\pi}{\sqrt{3RN}}}}
 =O\left(\sqrt{N}\right).
  \end{align*}
\end{proof}
To estimate the infinite product in $L_{a,c,d}(q)$ when $q$ is away from $1$, we need the following lemma.
\begin{lemma}\label{caw}
	If $y=\frac{1}{2\sqrt{3RN}}$ and $y\leq|x|\leq\frac{1}{2}$, then, as $N\rightarrow\infty$, we have
	\begin{align}
	\left | 	\frac{1}{(q^A;q^B)_\infty} \right |
	\ll \frac{\exp\left(\frac{C}{y}\right)}{(|q|^A;|q|^B)_\infty}, \label{leab}
\end{align} 
	where $A, B$ are relatively prime integers satisfying $B>A\geq 1$ and $C$ is a negative constant depending on $A, B$.
\end{lemma}
\begin{proof}
	We have
	\begin{align}
		\log\left(\frac{1}{(q^A;q^B)_\infty} \right)&=-\sum_{n=0}^\infty\log(1-q^{Bn+A})\no\\&=\sum_{n=0}^\infty\sum_{m=1}^{\infty}\frac{q^{Bmn+Am}}{m}\no\\&=\sum_{m=1}^\infty\frac{q^{Am}}{m(1-q^{Bm})}.\label{8121552}
	\end{align}
	
	Case (i):  $y\leq\mid x\mid \leq \frac{1}{4B}$. Note that
	\begin{align}
		\left | 	\log\left(\frac{1}{(q^A;q^B)_\infty} \right)\right |&=\left |\sum_{m=1}^\infty\frac{q^{Am}}{m(1-q^{Bm})}\right|\no\\
		&\leq \left|\frac{q^{A}}{(1-q^{B})}\right|+\left |\sum_{m=2}^\infty\frac{q^{Am}}{m(1-q^{Bm})}\right|\no
		\\
		&\leq\left|\frac{q^{A}}{(1-q^{B})}\right|-\frac{|q|^{A}}{(1-|q|^{B})}+\sum_{m=1}^\infty\frac{|q|^{Am}}{m(1-|q|^{Bm})}\no\\
		&\leq \left(\frac{1}{|(1-q^{B})|}-\frac{1}{(1-|q|^{B})}\right)+	\log\left(\frac{1}{(|q|^A;|q|^B)_\infty} \right).\label{x4b}
	\end{align}
Also, for $y\leq\mid x\mid \leq \frac{1}{4B}$,
	\begin{align*}
		|(1-q^{B})|^2&=(1-e^{2\pi iB (x+iy)})(1-e^{2\pi iB (-x+iy)})\\&=1-2e^{-2\pi By}\cos(2\pi Bx)+e^{-4\pi By}
		\\&\geq 1-2e^{-2\pi By}\cos(2\pi By)+e^{-4\pi By}
		\\&=1-2(1-2\pi By+2\pi^2 B^2y^2+O(y^3))\times(1-2\pi^2B^2y^2+O(y^4))\\&\quad+(1-4\pi By+8\pi^2B^2y^2+O(y^3))
		\\&=8\pi^2B^2y^2+O(y^3)
		\intertext{and}
		1-|q|^{B}&=2\pi By+O(y^2),
	\end{align*}
	as $y\rightarrow0^+.$
	Thus 
	\begin{align}
		\frac{1}{|(1-q^{B})|}-\frac{1}{(1-|q|^{B})}&\leq \frac{1+O(\sqrt{y})}{2\sqrt{2}\pi By}-\frac{1+O(y)}{2\pi By}\no
		\\ &=\frac{1-\sqrt{2}+O(\sqrt{y})}{2\sqrt{2}\pi By}.\label{x4b1}
	\end{align}
	Substituting \eqref{x4b1} into \eqref{x4b} yields
	\begin{align}
		\left | 	\log\left(\frac{1}{(q^A;q^B)_\infty} \right)\right |
		&\leq \frac{1-\sqrt{2}+O(\sqrt{y})}{2\sqrt{2}\pi By}+	\log\left(\frac{1}{(|q|^A;|q|^B)_\infty} \right).\no
	\end{align}
	It follows
	\begin{align}
		\left | 	\frac{1}{(q^A;q^B)_\infty} \right |
		&\leq 	\frac{\exp\left(\frac{-(\sqrt{2}-1)+O(\sqrt{y})}{2\sqrt{2}\pi By}\right)}{(|q|^A;|q|^B)_\infty} 
		\ll \frac{\exp\left(\frac{-(\sqrt{2}-1)}{4\pi By}\right)}{(|q|^A;|q|^B)_\infty}, \label{8121549}
	\end{align} 
when $y\leq\mid x\mid \leq \frac{1}{4B}$ and $y\rightarrow0^+.$
	
	Case (ii):  $\frac{1}{4B}\leq\mid x\mid \leq \frac{1}{2}$.
	
	Let 
	\begin{align*}
		P:&=\left\{x\,\Bigg|\, \Bigg.\frac{1}{4B}\leq\mid x\mid \leq \frac{1}{2}\right\}\\
		P_1:&=\left\{\frac{j}{B}\,\Bigg|\, \Bigg. j\in \mathbb{Z}\text{ and }1\leq |j|\leq \left\lceil\frac{B}{2}\right\rceil\right\}\\
		P_2:&=\left\{\frac{j}{B-A}\,\Bigg| \,\Bigg. j\in \mathbb{Z}\text{ and }1\leq |j|\leq \left\lceil\frac{B-A}{2}\right\rceil\right\}
	\end{align*}
	By assumption, we have $(B-A,B)=1$. Then $P_1\cap P_2=\emptyset$.
	Thus
	$$M^\prime:=\min\left\{|a-b|\Big| a\in P_1, b\in P_2\right\}>0.$$
	Let $M:=\min\left\{\frac{M^\prime}{4}, \frac{1}{8B}\right\}>0$ and define
	\begin{align*}
		P_3:&=\bigcup_{j\in\mathbb{Z}}  \left(\frac{j}{B}-M,\frac{j}{B}+M\right)\bigcap P\\
		P_4:&=\bigcup_{j\in\mathbb{Z}}  \left(\frac{j}{B-A}-M,\frac{j}{B-A}+M\right)\bigcap P.
	\end{align*}
	Then
	$$ P_3\cap P_4=\emptyset.$$
	When $x\in P\setminus P_3$, we have $\cos(2\pi Bx)\leq \cos (2\pi BM)<1$ and
	\begin{align*}
		|(1-q^{B})|^2&=1-2e^{-2\pi By}\cos(2\pi Bx)+e^{-4\pi By}
		\\&\geq 1-2e^{-2\pi By}\cos(2\pi BM)+e^{-4\pi By}
		\\&=2-2\cos(2\pi BM)+O(y)
		\\&\geq 8\pi^2B^2y^2+O(y^3),
	\end{align*}
	as $y\rightarrow0^+.$ Thus \eqref{8121549} remains true in this case.
	Then
	\begin{align}
		\left | 	\frac{1}{(q^A;q^B)_\infty} \right |
		&\leq 	\frac{\exp\left(\frac{-(\sqrt{2}-1)+O(\sqrt{y})}{2\sqrt{2}\pi By}\right)}{(|q|^A;|q|^B)_\infty} 
		\ll \frac{\exp\left(\frac{-(\sqrt{2}-1)}{4\pi By}\right)}{(|q|^A;|q|^B)_\infty}, \label{8132038}
	\end{align} 
	when $x\in P\setminus P_3$ and $y\rightarrow0^+.$ 
	
	Next, we consider $x\in P_3\subset P\setminus P_4$. 
	By \eqref{8121552}, we have
	\begin{align}
		\left | 	\log\left(\frac{1}{(q^A;q^B)_\infty} \right)\right |&=\left |\sum_{m=1}^\infty\frac{q^{Am}}{m(1-q^{Bm})}\right|\no\\
		&\leq \left|\frac{q^{A}}{(1-q^{B})}+\frac{q^{2A}}{2(1-q^{2B})}\right|+\left |\sum_{m=3}^\infty\frac{q^{Am}}{m(1-q^{Bm})}\right|\no
		\\
		&\leq \left|\frac{2q^{A}(1+q^{B})+q^{2A}}{2(1-q^{2B})}\right|-\frac{|q|^{A}}{(1-|q|^{B})}-\frac{|q|^{2A}}{2(1-|q|^{2B})}\no\\&\quad+\sum_{m=1}^\infty\frac{|q|^{Am}}{m(1-|q|^{Bm})}\no.
	\end{align}
	Note that
	\begin{align*}
		\left|2q^{A}(1+q^{B})+q^{2A}\right|&\leq  \left|2(1+q^{B})+q^{A}\right|\\
		&\leq 2+ \left|2q^{B}+q^{A}\right|
		\\
		&\leq 2+ \left|2q^{B-A}+1\right|.
	\end{align*}
	Then
	\begin{align}
		\left | 	\log\left(\frac{1}{(q^A;q^B)_\infty} \right)\right |
		&\leq \frac{ 2+ \left|2q^{B-A}+1\right|}{2(1-|q|^{2B})}-\frac{2|q|^{A}(1+|q|^{B})+|q|^{2A}}{2(1-|q|^{2B})}\no\\&\quad+\sum_{m=1}^\infty\frac{|q|^{Am}}{m(1-|q|^{Bm})}\no\\
		&= \frac{  \left|2q^{B-A}+1\right|-3+O(y)}{2(1-|q|^{2B})}+\log\left(\frac{1}{(|q|^A;|q|^B)_\infty} \right).\label{81310262}
	\end{align}
	Moreover, when $x\in  P\setminus P_4$, we have $\cos(2\pi (B-A)M)<1$ and
	\begin{align*}
		\left|2q^{B-A}+1\right|^2&=(1+2e^{2\pi i(B-A) (x+iy)})(1+2e^{2\pi i(B-A) (-x+iy)})\\&=1+4e^{-2\pi (B-A)y}\cos(2\pi (B-A)x)+4e^{-4\pi (B-A)y}
		\\&\leq1+4e^{-2\pi (B-A)y}\cos(2\pi (B-A)M)+4e^{-4\pi (B-A)y}
		\\&=5+4\cos(2\pi (B-A)M)+O(y)
	\end{align*}
	as $y\rightarrow0^+.$ This gives
	\begin{align}
		\left|2q^{B-A}+1\right|-3&=\frac{\left|2q^{B-A}+1\right|^2-9}{\left|2q^{B-A}+1\right|+3}\no\\&\leq\frac{ 4\left(\cos(2\pi (B-A)M)-1\right)+O(y)}{\left|2q^{B-A}+1\right|+3}
		\no	\\&\leq\frac{ 2\left(\cos(2\pi (B-A)M)-1\right)+O(y)}{3}.\label{81310261}
	\end{align}
	Substituting \eqref{81310261} in \eqref{81310262} yields 
	\begin{align}
		\left | 	\log\left(\frac{1}{(q^A;q^B)_\infty} \right)\right |
		&\leq \frac{  \left(\cos(2\pi (B-A)M)-1\right)+O(y)}{3(1-|q|^{2B})}+\log\left(\frac{1}{(|q|^A;|q|^B)_\infty} \right)\no\\
		&\leq \frac{  \left(\cos(2\pi (B-A)M)-1\right)+O(y)}{12\pi By}+\log\left(\frac{1}{(|q|^A;|q|^B)_\infty} \right).\no
	\end{align}
Thus
\begin{align}
	\left | 	\frac{1}{(q^A;q^B)_\infty} \right |
	&\leq 	\frac{\exp\left( \frac{  \left(\cos(2\pi (B-A)M)-1\right)+O(y)}{12\pi By}\right)}{(|q|^A;|q|^B)_\infty} 
	\ll 	\frac{\exp\left( \frac{  \left(\cos(2\pi (B-A)M)-1\right)}{24\pi By}\right)}{(|q|^A;|q|^B)_\infty}, \no
\end{align} 
	when $x\in P_3$ and $y\rightarrow0^+.$ This together with \eqref{8121549} and \eqref{8132038} proves 
	\eqref{leab}.
\end{proof}

Using Lemmas \ref{paw} and \ref{caw}, we get a bound for $L_{a,c,d}(q)$ in the region away from $1$. It is exponentially smaller than
the asymptotic discussed in Section \ref{ab}.
\begin{corollary}
If $y=\frac{1}{2\sqrt{3RN}}$ and $y\leq|x|\leq\frac{1}{2}$, then, as $N\rightarrow\infty$, we have
  \begin{align}
    |L_{a,c,d}(q)|=O\left(\sqrt{N}e^{\pi\sqrt{\frac{N}{3R}}+2C\sqrt{3RN}}\right),\label{8132150}
  \end{align}
with some negative constant $C$ depending only on $R, S$.
\end{corollary}
\begin{proof}
Lemmas \ref{paw} and \ref{caw} imply that, as $N\rightarrow\infty$, we have
  \begin{align}
 |L_{a,c,d}(q)|=   \left|\frac{G_{a,c,d}(q)}{(q^S,q^{R-S};q^R)_\infty}\right|\ll  \frac{\sqrt{N}\exp\left(\frac{C}{y}\right)}{(|q|^S,|q|^{R-S};|q|^R)_\infty}\label{8132145}
  \end{align}
with some negative constant $C$ depending on $R, S$.
 Applying \eqref{812914}, we find that
   \begin{align}
 \frac{1}{(|q|^S,|q|^{R-S};|q|^R)_\infty}=O\left(e^{\pi\sqrt{\frac{N}{3R}}}\right),\no
 \end{align}
which together with \eqref{8132145} gives \eqref{8132150}.
\end{proof}
\bigskip
\section{Wright’s CIRCLE METHOD}\label{zh2}
In this section, we apply Wright’s Circle Method
to complete the proof of Theorem \ref{thmain2}.

We give the detailed proof of \eqref{thmain2eq}.
By Cauchy's residue theorem, we have 
\begin{align}
\mathcal{B}_{a,c,d}(N)&=\frac{1}{2\pi i}\int_{\mathcal{C}}\frac{L_{a,c,d}(q)}{q^{N+1}}dq\no\\&=\int_{-\frac{1}{2}}^{\frac{1}{2}}
  L_{a,c,d}\left(e^{-\frac{\pi}{\sqrt{3RN}}+2\pi ix}\right)e^{\pi\sqrt{\frac{N}{3R}}-2\pi iNx}dx,\label{cc}
\end{align}
where the contour is the counterclockwise traversal of the circle $\mathcal{C}:=\left\{|q|=e^{-\frac{\pi}{\sqrt{3RN}}}\right\}$.
Separate the integral in \eqref{cc} into two ranges and write
\begin{align} \mathcal{B}_{a,c,d}(N)=I'+I''\label{8141117}
	\end{align}
 with
\begin{align*}
I':&=\int_{|x|\leq\frac{1}{2\sqrt{3RN}}}
  L_{a,c,d}\left(e^{-\frac{\pi}{\sqrt{3RN}}+2\pi ix}\right)e^{\pi\sqrt{\frac{N}{3R}}-2\pi iNx}dx\\
  \intertext{and}\\
 I'':&=\int_{|x|\leq\frac{1}{2\sqrt{3RN}}\leq\frac{1}{2}}
 L_{a,c,d}\left(e^{-\frac{\pi}{\sqrt{3RN}}+2\pi ix}\right)e^{\pi\sqrt{\frac{N}{3R}}-2\pi iNx}dx .
\end{align*}
We prove that $I'$ contributes the main terms in the asymptotic expansion of $\mathcal{B}_{a,c,d}(N)$ 
and the integral $I''$ is absorbed into the error term.
\subsection{Main arc.}
For a real number $s$, define
\begin{align*}
  P_s:=\frac{1}{2\pi i}\int_{1-i}^{1+i}v^se^{\pi\sqrt{\frac{N}{3R}}(v+\frac{1}{v})}dv.
\end{align*}
By \cite[Lemma 4.2]{kk}, as $N\rightarrow\infty$,
we have
\begin{align}
  P_s-I_{-s-1}\left(2\pi\sqrt{\frac{N}{3R}}\right)=O\left(e^{\frac{3\pi}{2}\sqrt{\frac{N}{3R}}}\right).\label{8141100}
\end{align}
Applying \eqref{8141100}, we evaluate $I'$ by the modified Bessel functions up to an allowable error.
\begin{proposition}\label{p1}
As $N\rightarrow\infty$, we have
\begin{align}
I'
&=\frac{\sqrt{\frac{\pi}{a}}}{4\sin\left(\frac{S\pi }{R}\right)}
\times\left(\frac{\pi}{\sqrt{3RN}}\right)^{1/2}
I_{-\frac{1}{2}}\left(2\pi\sqrt{\frac{N}{3R}}\right)\nonumber\\
&\quad-\frac{B_1\left(\frac{c}{2a}\right)}{2\sin\left(\frac{S\pi }{R}\right)}\left(\frac{\pi}{\sqrt{3RN}}\right)
I_{-1}\left(2\pi\sqrt{\frac{N}{3R}}\right)\nonumber\\
&\quad -\frac{\sqrt{\frac{\pi}{a}}\left(d-\frac{c^2}{4a}+\frac{R}{12}-\frac{S}{2}+\frac{ S^2}{2R} \right)}{4\sin\left(\frac{S\pi }{R}\right)}
\left(\frac{\pi}{\sqrt{3RN}}\right)^{3/2}
I_{-\frac{3}{2}}\left(2\pi\sqrt{\frac{N}{3R}}\right)\nonumber\\
&\quad +\frac{\left[	\left(d-\frac{c^2}{4a}+\frac{R}{12}-\frac{S}{2}+\frac{ S^2}{2R} \right)B_1\left(\frac{c}{2a}\right)+\frac{aB_3\left(\frac{c}{2a}\right)}{3}\right]}{2\sin\left(\frac{S\pi }{R}\right)}
\left(\frac{\pi}{\sqrt{3RN}}\right)^{2}
I_{-2}\left(2\pi\sqrt{\frac{N}{3R}}\right)
\nonumber\\
&\quad
+O\left(N^{-\frac{3}{2}}e^{2\pi\sqrt{\frac{N}{3R}}}\right)
	.\label{8141059}
\end{align}
\end{proposition}
\begin{proof}
Writing $\tau=\frac{1}{2\sqrt{3RN}}(u+i)$, i.e., replacing $x$ by $\frac{u}{2\sqrt{3RN}}$ yields
\begin{align}
	I'&=\int_{|x|\leq\frac{1}{2\sqrt{3RN}}}
	L_{a,c,d}\left(e^{-\frac{\pi}{\sqrt{3RN}}+2\pi ix}\right)e^{\pi\sqrt{\frac{N}{3R}}-2\pi iNx}dx\nonumber\\
	&=\frac{1}{2\sqrt{3RN}}\int_{-1}^{1}
	L_{a,c,d}\left(e^{\frac{\pi}{\sqrt{3RN}}(-1+iu)}\right)e^{\pi\sqrt{\frac{N}{3R}}(1-iu)}du. \label{ci11}
\end{align}
Applying \eqref{8131410} with 
 $\tau$ replaced by $\frac{1}{2\sqrt{3RN}}(u+i)$ and
noting
that $-2\pi i\tau=\frac{\pi(1-iu)}{\sqrt{3RN}}$ and $e^{\frac{\pi i}{6R\tau}}=e^{\pi\sqrt{\frac{N}{3R}}\left(\frac{1}{1-iu}\right)}$, we have
\begin{align}
 & L_{a,c,d}\left(e^{\frac{\pi}{\sqrt{3RN}}(-1+iu)}\right)\no\\&=\no\frac{\sqrt{\frac{\pi}{a}}e^{\pi\sqrt{\frac{N}{3R}}\left(\frac{1}{1-iu}\right)}}{4\sin\left(\frac{S\pi }{R}\right)}\left(\frac{\pi(1-iu)}
  {\sqrt{3RN}}\right)^{-\frac{1}{2}}-	\frac{B_1\left(\frac{c}{2a}\right)e^{\pi\sqrt{\frac{N}{3R}}\left(\frac{1}{1-iu}\right)}}{2\sin\left(\frac{S\pi }{R}\right)}\\&\quad-\frac{\sqrt{\frac{\pi}{a}}\left(d-\frac{c^2}{4a}+\frac{R}{12}-\frac{S}{2}+\frac{ S^2}{2R} \right)e^{\pi\sqrt{\frac{N}{3R}}\left(\frac{1}{1-iu}\right)}}{4\sin\left(\frac{S\pi }{R}\right)}\left(\frac{\pi(1-iu)}
  {\sqrt{3RN}}\right)^{\frac{1}{2}}
    \nonumber\\ &\quad +\frac{\left[	\left(d-\frac{c^2}{4a}+\frac{R}{12}-\frac{S}{2}+\frac{ S^2}{2R} \right)B_1\left(\frac{c}{2a}\right)+\frac{aB_3\left(\frac{c}{2a}\right)}{3}\right]e^{\pi\sqrt{\frac{N}{3R}}\left(\frac{1}{1-iu}\right)}}{2\sin\left(\frac{S\pi }{R}\right)}\left(\frac{\pi(1-iu)}
    {\sqrt{3RN}}\right)
  \nonumber\\ &\quad 
 +\zeta_1^\ast\left(\frac{\pi(1-iu)}
  {\sqrt{3RN}}\right)^{3/2} e^{\pi\sqrt{\frac{N}{3R}}\left(\frac{1}{1-iu}\right)} +\zeta_2^\ast\left(\frac{\pi(1-iu)}
  {\sqrt{3RN}}\right)^{2} e^{\pi\sqrt{\frac{N}{3R}}\left(\frac{1}{1-iu}\right)}\no\\&\quad+\zeta_3^\ast\left(\frac{\pi(1-iu)}
  {\sqrt{3RN}}\right)^{5/2} e^{\pi\sqrt{\frac{N}{3R}}\left(\frac{1}{1-iu}\right)}+O\left(e^{\pi\sqrt{\frac{N}{3R}}}N^{-\frac{3}{2}}\right)\quad(\textrm{as } N\rightarrow\infty). \label{cfkl}
\end{align}
	where $\zeta_1^\ast,\zeta_2^\ast,\zeta_3^\ast$ are constants depending on $R, S, a, c,$ and $d$.
 Substituting \eqref{cfkl} into \eqref{ci11}
gives
\begin{align*}
I'
&=\frac{\sqrt{\frac{\pi}{a}}}{8\sqrt{3RN}\sin\left(\frac{S\pi }{R}\right)}\int_{-1}^{1}
e^{\pi\sqrt{\frac{N}{3R}}\left(\frac{1}{1-iu}+(1-iu)\right)}\left(\frac{\pi(1-iu)}
{\sqrt{3RN}}\right)^{-\frac{1}{2}}du\\
  &\quad-\frac{B_1\left(\frac{c}{2a}\right)}{4\sqrt{3RN}\sin\left(\frac{S\pi }{R}\right)}\int_{-1}^{1}
  e^{\pi\sqrt{\frac{N}{3R}}\left(\frac{1}{1-iu}+(1-iu)\right)}du\\&\quad
  -\frac{\sqrt{\frac{\pi}{a}}\left(d-\frac{c^2}{4a}+\frac{R}{12}-\frac{S}{2}+\frac{ S^2}{2R} \right)}{8\sqrt{3RN}\sin\left(\frac{S\pi }{R}\right)}\int_{-1}^{1}
  e^{\pi\sqrt{\frac{N}{3R}}\left(\frac{1}{1-iu}+(1-iu)\right)}\left(\frac{\pi(1-iu)}
  {\sqrt{3RN}}\right)^{\frac{1}{2}}du\\
   &\quad+\frac{\left[	\left(d-\frac{c^2}{4a}+\frac{R}{12}-\frac{S}{2}+\frac{ S^2}{2R} \right)B_1\left(\frac{c}{2a}\right)+\frac{aB_3\left(\frac{c}{2a}\right)}{3}\right]}{4\sqrt{3RN}\sin\left(\frac{S\pi }{R}\right)}
   \\&\quad\times\int_{-1}^{1}
  e^{\pi\sqrt{\frac{N}{3R}}\left(\frac{1}{1-iu}+(1-iu)\right)}\left(\frac{\pi(1-iu)}{\sqrt{3RN}}\right)du\\
  &\quad+\frac{\zeta_1^\ast}{2\sqrt{3RN}}\int_{-1}^{1}
  e^{\pi\sqrt{\frac{N}{3R}}\left(\frac{1}{1-iu}+(1-iu)\right)}\left(\frac{\pi(1-iu)}{\sqrt{3RN}}\right)^{3/2}du
  \\
  &\quad+\frac{\zeta_2^\ast}{2\sqrt{3RN}}\int_{-1}^{1}
  e^{\pi\sqrt{\frac{N}{3R}}\left(\frac{1}{1-iu}+(1-iu)\right)}\left(\frac{\pi(1-iu)}{\sqrt{3RN}}\right)^{2}du
  \\
  &\quad+\frac{\zeta_3^\ast}{2\sqrt{3RN}}\int_{-1}^{1}
  e^{\pi\sqrt{\frac{N}{3R}}\left(\frac{1}{1-iu}+(1-iu)\right)}\left(\frac{\pi(1-iu)}{\sqrt{3RN}}\right)^{5/2}du
 \\&\quad +\int_{-1}^{1}O\left(N^{-\frac{3}{2}}e^{\pi\sqrt{\frac{N}{3R}}}\right)
  e^{\pi\sqrt{\frac{N}{3R}}(1-iu)}du.
\end{align*}
as $N\rightarrow\infty$.
Making the change of variables $u=i(v-1)$, we find that
\begin{align*}
&\frac{1}{2\sqrt{3RN}}\int_{-1}^{1}
  \left(\frac{\pi(1-iu)}{\sqrt{3RN}}\right)^{s}e^{\pi\sqrt{\frac{N}{3R}}\left(\frac{1}{1-iu}+(1-iu)\right)}du\\&=
\frac{i}{2\sqrt{3RN}}\int_{1+i}^{1-i}\left(\frac{\pi v}{\sqrt{3RN}}\right)^se^{\pi\sqrt{\frac{N}{3R}}(v+\frac{1}{v})}dv\\&=
\frac{-i}{2\sqrt{3RN}}\left(\frac{\pi }{\sqrt{3RN}}\right)^s\int_{1-i}^{1+i}v^se^{\pi\sqrt{\frac{N}{3R}}(v+\frac{1}{v})}dv
 \\
 &=\left(\frac{\pi}{\sqrt{3RN}}\right)^{s+1}P_s\\&=\left(\frac{\pi}{\sqrt{3RN}}\right)^{s+1}
 I_{-s-1}\left(2\pi\sqrt{\frac{N}{3R}}\right)+O\left(N^{-(s+1)/2}e^{\frac{3\pi}{2}\sqrt{\frac{N}{3R}}}\right)\qquad\textrm{ (by \eqref{8141100}) } .
\end{align*}
It follows
\begin{align}
I'
&=\frac{\sqrt{\frac{\pi}{a}}}{4\sin\left(\frac{S\pi }{R}\right)}
\times\left(\frac{\pi}{\sqrt{3RN}}\right)^{1/2}
 I_{-\frac{1}{2}}\left(2\pi\sqrt{\frac{N}{3R}}\right)\nonumber\\
  &\quad-\frac{B_1\left(\frac{c}{2a}\right)}{2\sin\left(\frac{S\pi }{R}\right)}\left(\frac{\pi}{\sqrt{3RN}}\right)
 I_{-1}\left(2\pi\sqrt{\frac{N}{3R}}\right)\nonumber\\
 &\quad -\frac{\sqrt{\frac{\pi}{a}}\left(d-\frac{c^2}{4a}+\frac{R}{12}-\frac{S}{2}+\frac{ S^2}{2R} \right)}{4\sin\left(\frac{S\pi }{R}\right)}
 \left(\frac{\pi}{\sqrt{3RN}}\right)^{3/2}
 I_{-\frac{3}{2}}\left(2\pi\sqrt{\frac{N}{3R}}\right)\nonumber\\
  &\quad +\frac{\left[	\left(d-\frac{c^2}{4a}+\frac{R}{12}-\frac{S}{2}+\frac{ S^2}{2R} \right)B_1\left(\frac{c}{2a}\right)+\frac{aB_3\left(\frac{c}{2a}\right)}{3}\right]}{2\sin\left(\frac{S\pi }{R}\right)}
 \left(\frac{\pi}{\sqrt{3RN}}\right)^{2}
 I_{-2}\left(2\pi\sqrt{\frac{N}{3R}}\right)\nonumber\\
 &\quad+\zeta_1^\ast  \left(\frac{\pi}{\sqrt{3RN}}\right)^{5/2}
 I_{-\frac{5}{2}}\left(2\pi\sqrt{\frac{N}{3R}}\right)
 +\zeta_2^\ast  \left(\frac{\pi}{\sqrt{3RN}}\right)^{3}
 I_{-3}\left(2\pi\sqrt{\frac{N}{3R}}\right)
 \nonumber\\
 &\quad+\zeta_3^\ast  \left(\frac{\pi}{\sqrt{3RN}}\right)^{7/2}
 I_{-\frac{7}{2}}\left(2\pi\sqrt{\frac{N}{3R}}\right)
  +O\left(N^{-\frac{3}{2}}e^{2\pi\sqrt{\frac{N}{3R}}}\right)\quad(\textrm{as } N\rightarrow\infty)
.\label{ppi}
\end{align}
By \eqref{eqi}, we have
\begin{align*}
 I_{-3}\left(2\pi\sqrt{\frac{N}{3R}}\right)=
  O\left(N^{-1/4}e^{2\pi\sqrt{\frac{N}{3R}}}\right)\quad(\textrm{as } N\rightarrow\infty)
  ,
\end{align*}
which together with \eqref{ppi} gives \eqref{8141059}.
\end{proof}
\subsection{Error arc.}
We give a bound for $I''$ which is exponentially smaller than the error term of $I'$.
\begin{proposition}\label{p2}
  As $N\rightarrow\infty$,
  \begin{align*}
    I''=O\left(\sqrt{N}e^{2\pi\sqrt{\frac{N}{3R}}+2C\sqrt{3RN}}\right)
  \end{align*}
with some negative constant $C$ depending only on $R, S$.
\end{proposition}
\begin{proof}
  By \eqref{8132150}, as $N\rightarrow\infty$, we have
\begin{align*}
| I''|&=\left|\int_{|x|\leq\frac{1}{2\sqrt{3RN}}\leq\frac{1}{2}}
  L_{a,c,d}\left(e^{-\frac{\pi}{\sqrt{3RN}}+2\pi ix}\right)e^{\pi\sqrt{\frac{N}{3R}}-2\pi iNx}dx\right|
  \\&= O\left(\sqrt{N}e^{\pi\sqrt{\frac{N}{3R}}+2C\sqrt{3RN}}\right)\times\left|\int_{|x|\leq\frac{1}{2\sqrt{3RN}}\leq\frac{1}{2}}
  e^{\pi\sqrt{\frac{N}{3R}}-2\pi iNx}dx\right|\\&
  =O\left(\sqrt{N}e^{2\pi\sqrt{\frac{N}{3R}}+2C\sqrt{3RN}}\right)
\end{align*}
with some negative constant $C$ depending only on $R, S$.
\end{proof}
\subsection{Proof of \eqref{thmain2eq}.}
Invoking Propositions \ref{p1} and \ref{p2} in equation \eqref{8141117}, we find that, as $N\rightarrow\infty$,
\begin{align*}
 \mathcal{B}_{a,c,d}(N)&=I'+I''
 =\frac{\sqrt{\frac{\pi}{a}}}{4\sin\left(\frac{S\pi }{R}\right)}
 \times\left(\frac{\pi}{\sqrt{3RN}}\right)^{1/2}
 I_{-\frac{1}{2}}\left(2\pi\sqrt{\frac{N}{3R}}\right)\nonumber\\
 &\quad-\frac{B_1\left(\frac{c}{2a}\right)}{2\sin\left(\frac{S\pi }{R}\right)}
\left(\frac{\pi}{\sqrt{3RN}}\right) I_{-1}\left(2\pi\sqrt{\frac{N}{3R}}\right)\nonumber\\
 &\quad -\frac{\sqrt{\frac{\pi}{a}}\left(d-\frac{c^2}{4a}+\frac{R}{12}-\frac{S}{2}+\frac{ S^2}{2R} \right)}{4\sin\left(\frac{S\pi }{R}\right)}
 \left(\frac{\pi}{\sqrt{3RN}}\right)^{3/2}
 I_{-\frac{3}{2}}\left(2\pi\sqrt{\frac{N}{3R}}\right)\nonumber\\
 &\quad +\frac{\left[	\left(d-\frac{c^2}{4a}+\frac{R}{12}-\frac{S}{2}+\frac{ S^2}{2R} \right)B_1\left(\frac{c}{2a}\right)+\frac{aB_3\left(\frac{c}{2a}\right)}{3}\right]}{2\sin\left(\frac{S\pi }{R}\right)}
 \left(\frac{\pi}{\sqrt{3RN}}\right)^{2}
 I_{-2}\left(2\pi\sqrt{\frac{N}{3R}}\right)\nonumber\\
 &\quad
 +O\left(N^{-\frac{3}{2}}e^{2\pi\sqrt{\frac{N}{3R}}}\right)+O\left(\sqrt{N}e^{2\pi\sqrt{\frac{N}{3R}}+2C\sqrt{3RN}}\right).
\end{align*}
Since $C<0$, we have $$\sqrt{N}e^{2\pi\sqrt{\frac{N}{3R}}+2C\sqrt{3RN}}\ll N^{-3/2}e^{2\pi\sqrt{\frac{N}{3R}}}$$
and \eqref{thmain2eq} follows.

\subsection{Sketch of proof \eqref{thmain2eqp}.}
We omit the proof of \eqref{thmain2eqp} and only briefly mention that the proof
is completely similar to that of \eqref{thmain2eq}. The only difference one should take care is the asymptotic expansion of the infinite product $\frac{1}{(q^S,q^{R-S},q^R;q^R)_\infty}$.
Using \eqref{812trtha} and proceeding as in the proof of \eqref{812914}, one can find an analog of Lemma \ref{led2} for $\frac{1}{(q^S,q^{R-S},q^R;q^R)_\infty}$. 
Then asymptotics for the generating function $L^\prime_{a,c,d}(q)$ can be found as in Section \ref{zh1}. We have
\begin{enumerate}[(i)]
	\item for $y=\frac{1}{2\sqrt{2RN}}$ and $|x|\leq y$, as $N\rightarrow\infty$, 
\begin{align}
	L^\prime_{a,c,d}(q)&=	\frac{e^{\frac{\pi i}{4R\tau}}}{2\sin\left(\frac{S\pi }{R}\right)}
	\times \Bigg\{\frac{1}{2}\sqrt{\frac{R}{2a}}-\sqrt{\frac{R}{2\pi}}B_1\left(\frac{c}{2a}\right)	(-2\pi i\tau)^{\frac{1}{2}}\no\\&\qquad\qquad-\frac{	\left(d-\frac{c^2}{4a}+\frac{R}{8}-\frac{S}{2}+\frac{ S^2}{2R} \right)}{2}\sqrt{\frac{R}{2a}}	(-2\pi i\tau)\no\\&\qquad\qquad +\left[	\left(d-\frac{c^2}{4a}+\frac{R}{8}-\frac{S}{2}+\frac{ S^2}{2R} \right)B_1\left(\frac{c}{2a}\right)+\frac{aB_3\left(\frac{c}{2a}\right)}{3}\right]\sqrt{\frac{R}{2\pi}}(-2\pi i\tau)^\frac{3}{2}\no\\&\qquad\qquad+\zeta_1^\ast\tau^{2}+\zeta_2^\ast\tau^{5/2}+\zeta_3^\ast\tau^{3}\Bigg\}+O\left(e^{\pi\sqrt{\frac{N}{3R}}}N^{-\frac{7}{4}}\right), \label{1627lp}
\end{align}
where $\zeta_1^\ast,\zeta_2^\ast,\zeta_3^\ast$ are constants depending on $R, S, a, c$ and $d$.

	\item
for $y=\frac{1}{2\sqrt{2RN}}$ and $y\leq|x|\leq\frac{1}{2}$, then, as $N\rightarrow\infty$, 
	\begin{align}
		|L^\prime_{a,c,d}(q)|=O\left(N^{1/4}e^{\pi\sqrt{\frac{N}{2R}}+2C\sqrt{2RN}}\right),\label{1627lp1}
	\end{align}
	with some negative constant $C$ depending only on $R, S$.
\end{enumerate}

 At last, armed with \eqref{1627lp} and \eqref{1627lp1}, it is routine to use Wright’s Circle Method to obtain \eqref{thmain2eqp}.

\bigskip

\textbf{Acknowledgements.} 
The author thanks Professor Sun Lisa Hui  for her valuable suggestions.
This work is supported by NSFC (National Natural Science
Foundation of China) through Grants NSFC 12071331.

\end{document}